\documentclass{ws-jaa}
\usepackage{amsmath,amsfonts,amssymb,  xfrac}
\usepackage{enumerate}
\usepackage{mathrsfs}
\usepackage{graphicx}
\usepackage{tikz}
\usepackage{float}
\usepackage{xcolor}
\usepackage[verbose]{hyperref}
\hypersetup{colorlinks=false,allbordercolors=blue,pdfborderstyle={/S/U/W 1}}

\newcommand{\N}{\mathbb{N}}
\newcommand{\Z}{\mathbb{Z}}

\newcommand{\F}{\mathbb{F}}
\newcommand{\Fq}{\mathbb{F}_q}
\newcommand{\Fqd}{\mathbb{F}_{q^2}}
\newcommand{\ord}{\mathop{\rm ord}}
\newcommand{\lcm}{\mathop{\rm lcm}}

\newcommand{\D}{\cdots}
\newcommand{\T}{\mathscr{T}}

\newcommand{\G}{\mathscr{G}}
\newcommand{\V}{\mathcal{V}}
\newcommand{\E}{\mathcal{E}}
\newcommand{\nup}{\nu_p}

\newtheorem{Th}{Theorem}
\newtheorem{Co}[Th]{Corollary}

\newtheorem{Df}[Th]{Definition}
\newtheorem{Le}[Th]{Lemma}
\newtheorem{Ex}[Th]{Example}
\newtheorem{Obs}[Th]{Remark}

\begin{document}

\title{The dynamics of a function family over quadratic extensions of finite fields}

\author{Hugo Teixeira}
\address{School of Mathematics and Statistics, Carleton University, Canada\\
\email{hugoteixeira@cmail.carleton.ca}}

\author{Fabio Brochero Martínez}
\address{Departamento de Matemática, UFMG, Brasil\\
\email{fbrocher@mat.ufmg.br}}


\maketitle
\date{\today}

\begin{abstract}
    Let $\Fq$ be the finite field with $q=p^s$ elements, where $p$ is an odd prime and $s$ a positive integer. In this paper, we analyze the function $f(X)=(cX^q+aX)(X^{q}-X)^{n-1}$, for $a,c\in\Fq$ and $n\geq 1$. Viewing $\Fqd$ as a two-dimensional vector space over $\Fq$, we obtain an explicit algebraic description of the induced dynamical system.
    We determine all possible cycle lengths, the exact number of cycles of each length, and give a complete classification of the trees attached to periodic points.
    The structure of the functional graph is shown to depend explicitly on arithmetic invariants of $\Fq$, including greatest common divisors and multiplicative characters.
    This provides a complete description of the functional graph of $f$ over $\Fqd$.
    
\end{abstract}
\section{Introduction}

Let $\Fq$ be the finite field with $q=p^s$ elements, where $p$ is an odd prime and $s$ is a positive integer. For each arbitrary function $f:\Fq\to \Fq$, the repeated iteration of $f(X)$ gives us a dynamical system that can be represented as a directed graph. We define the {\em functional graph of $f$} over $\Fq$ as the directed graph $\G(f)=(\V, \E)$, where the set of vertices is $\V=\Fq$ and the set of directed edges is $\E=\{\langle x,f(x)\rangle\mid x\in X\}$.

In addition, we define the iterations of $f$ as $f^{(0)}(X)=X$ and, for any $n>0$, $f^{(n+1)}(X)=f(f^{(n)}(X))$. Since $f$ is defined over a finite set, for a fixed $a\in \Fq$, there are minimal positive integers $0\leq i<j$,  such that $f^{(i)}(a)=f^{(j)}(a)$. When $i>0$ we call the list $$a,f(a),f^{(2)}(a),\D,f^{(i-1)}(a)$$ a pre-cycle and $$f^{(i)}(a),f^{(i+1)}(a),\D,f^{(j-1)}(a)$$ a cycle of length $(j-i)$ or $(j-i)$-cycle. If $a$ is an element of a cycle it is a {\it periodic element} and if $f(a)=a$ it is a {\it fixed point}.

\begin{Ex}
    Let $q=25$ and $f:\F_{25}\to \F_{25}$ be the function defined by $f(x)=x^6+x^2+1$, then the dynamics of $f$ is represented by the graph
\begin{figure}[H]
	\begin{center}
		\includegraphics[scale=0.3]{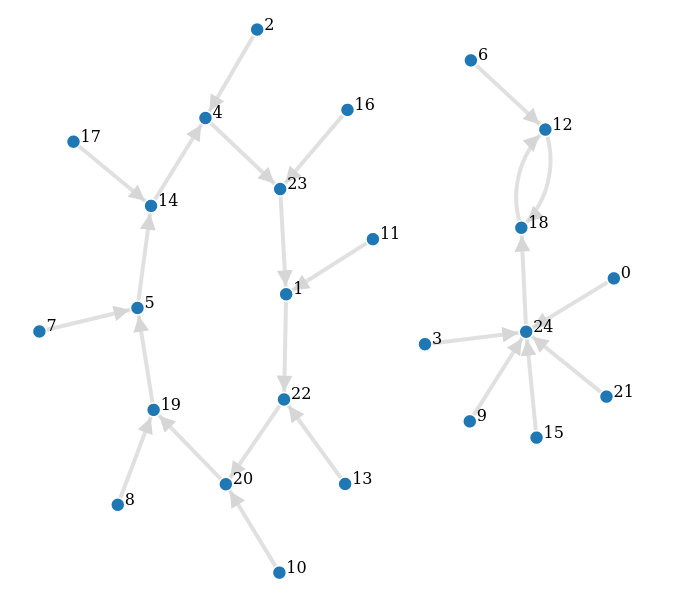}
	\end{center}
	\caption{Functional graph of $f(x)=x^6+x^2+1$ over $\F_{25}$.}
\end{figure}

 where the vertices are labeled by the positive exponent of their representation as a power of a generator of the multiplicative group $\Fq^*$, and the element $0$ is labeled by $0$. 

 We highlight that in this case the vertex $4$ is a periodic point contained in a cycle of length $8$ and the vertex $2$ is in the pre-cycle of the vertex $4$.
\end{Ex}

The main interest in functional graphs is based in their computational applications. For example,  M. Blum, L. Blum and M. Shub \cite{BBS} use the map $f(X)=X^2$ over the finite ring $\Z_n$, where $n$ is the product of two distinct primes, to generate pseudorandom sequences over $\Z_2$. Pseudorandom sequences are sequences of numbers that are computationally difficult to predict. Another application is in cryptography, as seen in \cite{NySo}, where D. Nyang and J. Song use the map $f(X)=X^2$ over the finite ring $\Z_{pq}$ to produce a signature scheme. The safety of both of these application rely on the fact that the product of two large primes cannot be factored in a polynomial time.

The properties of the dynamical system generated by a function, such as number of cycles, cycle lengths and pre-cycle lengths have been studied for several different maps over finite fields and rings. Maps of the form $x^2+c$ have already been thoroughly described as seen in \cite{PeMoMuYu,Rog1,VaSh}. This quadratic function has a highly predictable behavior in the case when $c=0$ or $c=2$, and essentially a random behavior in the other cases \cite{Ei-Eh}. The dynamics of other remarkable functions such as repeated exponentiation \cite{ChSh}, Chebyshev polynomials \cite{Gas,QurPan2} and Rédei functions \cite{QurPan1} have already been described.

In this paper we present the complete characterization of the digraph associated to the map $f(X)=(cX^q+aX)(X^q-X)^{n-1}$, for $a, c\in \Fq$, $n\ge 2$ over $\F_{q^2}$. We determine the number of fixed points, the number of cycles of each length as well as the precise behavior of the pre-cycles. From an algebraic point of view, the map $f$ can be seen as a product of $\Fq$-linear components, which allows the dynamics over $\Fqd$ to be analyzed via its vector space structure over $\Fq$. In particular, when $n=2$  this function defines a quadratic form over $\Fq$, which was studied by the authors in \cite{BrTx} using $c=0$. In addition, it is a natural generalization of the function $x\mapsto a x^2$ over $\Fq$. 

The main contribution of this work is a complete classification of the functional graph of $f$ over $\Fqd$, including all periodic components and the precise structure of the trees attached to cyclic elements. This paper is divided in three sections. Section 2 contains the description of the functional graph in the case when $n$ is even and Section 3 contains the description of the functional graph in the case when $n$ is odd.

\section{The dynamics for $n$ even}

Let $a$ and $c$ be elements in $\Fq$ and $n$ be a positive integer. We define the function $$f(X)=(cX^q+aX)(X^q-X)^{n-1}.$$ We notice that the case $c=a=0$ is trivial, therefore, we assume that at least one of them is not zero. In this section, we analyze the case where $n$ is even.

To understand the behavior of $f(X)$ we look at it in a different perspective. We have that $\Fqd$ may be seen as a vector space over $\Fq$. Thus, we look at $f(X)$ as a map of vector spaces defined over $\Fq\times \Fq$.

Let $\beta\in\Fqd$ be such that $\beta\notin\Fq$ and $\beta^2\in\Fq$. We have that $\{1,\beta\}$ forms  a base of $\Fqd$ over $\Fq$. Moreover, $\beta^q=b^{\frac{q-1}{2}}\beta=-\beta$ and $\beta^{q+1}=-\beta^2$.

 Writing $X\in\Fqd$ as $X=x+y\beta$, with $x,y\in\Fq$, we obtain
\begin{align*}
    f(X)&=(cX^q+aX)(X^{q}-X)^{n-1}\\
    &=(cx-cy\beta+ax+ay\beta)(x-y\beta-x-y\beta)^{n-1}\\
    &=((a+c)x+(a-c)y\beta)(-2y\beta)^{n-1}\\
    &=(a-c)(-2)^{n-1}\beta^ny^n+(a+c)(-2)^{n-1}(\beta)^{n-2}xy^{n-1}\beta
\end{align*}

Since $n$ is even, $\beta^n$ and $\beta^{n-2}$ are elements of $\Fq$, then, defining $\delta_1=(a-c)(-2)^{n-1}\beta^n$ and $\delta_2=(a+c)(-2)^{n-1}(\beta)^{n-2}$, we have

\begin{align}
    \Fq\times\Fq&\overset{f}{\to} \Fq\times\Fq\nonumber\\
    \langle x,y\rangle&\mapsto\langle\delta_1y^n,\delta_2xy^{n-1}\rangle\label{f1}
\end{align}

With this notation, it is verified directly that $\langle0,0\rangle$ is a fixed point. Our first goal is to describe the connected component that contains the point $\langle0,0\rangle$.

\subsection{The zero component}

\begin{Df}
For each element $\alpha\in\Fqd$, we define  $f^{-1}(\alpha)=\{x\in\Fqd\mid f(x)=\alpha\}$. We say it is the preimage of $\alpha$.
\end{Df}

We observe that, if $\langle x,y\rangle\in f^{-1}(0,0)$, then the following equations must be satisfied $$\delta_1y^n=0\ \text{and}\ \delta_2xy^{n-1}=0.$$ The following theorems analyse the cases when $\delta_1=0$ or $\delta_2=0$.

\begin{Th}\label{Theodelta1=0}
    If $\delta_1=0$, then the functional graph of $f$ has one connected component. Moreover, it is composed by a cycle of length one with $2q-2$ vertices directed to it and, amongst those, $q-1$ have $q-1$ elements in its preimage. 
\end{Th}


%

\begin{proof}
    If $\delta_1=0$, then $\delta_2\neq0$ by definition and we conclude that, for every $u\in\Fq$, $\langle u,0\rangle$ and $\langle 0,u\rangle$ are the only elements directed to $\langle0,0\rangle$. Then $$\#f^{-1}(0,0)=2q-1.$$ We notice that $\langle a,0\rangle$ has no preimage for every $a\in\Fq^*$. On the other hand, using  (\ref{f1}) we can conclude that the preimage of $\langle 0,u\rangle$ is composed by all the elements defined by $\langle \frac{u}{\delta_2y^{n-1}},y\rangle$, where $y\in\Fq^*$, and all of them have empty preimage. 
    
    Therefore $\#f^{-1}(0,u)=q-1$ and, since there are $q-1$ choices for $u$, we have $(q-1)(q-1)+2q-1=q^2$ elements in this connected component, thus it is the only component of this functional graph.
\end{proof}

\begin{Ex}
Taking $q=7$ and $f(X)=(3X^7+3X)(X^7-X)$, we have $\delta_1=0$ and the functional graph of $f(X)$ is isomorphic to 

\begin{figure}[H]
	\begin{center}
	\includegraphics[scale=0.4]{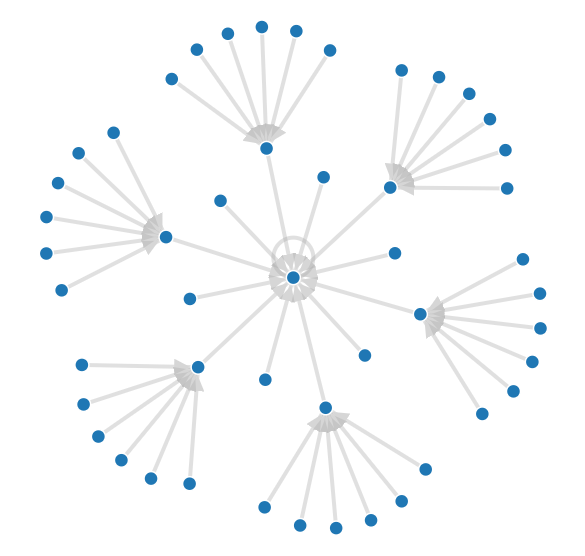}
	\end{center}
	\caption{Functional graph of $f(X)=(3X^7+3X)(X^7-X)$ over $\F_{49}$.}
\end{figure}
\end{Ex}

\begin{Df}
For $m\in\Z$ we define $\mathfrak{g}(m)=\gcd(m,q-1)$.
\end{Df}

\begin{Th}\label{Theodelta2=0}
    If $\delta_2=0$, then the functional graph of $f$ has one connected component. Moreover, it is composed by one cycle of length one with $q-1$ vertices directed to it and, amongst those, $\frac{q-1}{\mathfrak{g(n)}}$ have $q\mathfrak{g}(n)$ vertices in its preimage
\end{Th}
\begin{proof}
    If $\delta_2=0$, then $\delta_1\neq 0$ and the only elements in $f^{-1}(0,0)$ are $\langle u,0\rangle$, for every $u\in\Fq^*$. This implies that $\#f^{-1}(0,0)=q$. We observe that $\langle u,0\rangle$ has nonempty preimage if and only if $\frac{u}{\delta_1}$ is a $n$th power in $\Fq$. Furthermore, the elements in its preimage are of the form $\langle v, \lambda\rangle$, where $\lambda^n=\frac{u}{\delta_1}$ and $v$ is any element of $\Fq$. Since there are $q\cdot \mathfrak{g}(n)$ elements of such form and $\frac{q-1}{\mathfrak{g}(n)}$ elements that are $n$th powers in $\Fq$, we have $q\cdot \mathfrak{g}(n)\cdot \frac{q-1}{\mathfrak{g}(n)}+q=q^2$ elements in this connected components, thus it is the only connected component in this functional graph.
\end{proof}

\begin{Ex}
Now, taking $q=7$ and $f(X)=(3X^7+4X)(X^7-X)$, we have $\delta_2=0$ and the following graph
\begin{figure}[H]
	\begin{center}
		\includegraphics[scale=0.4]{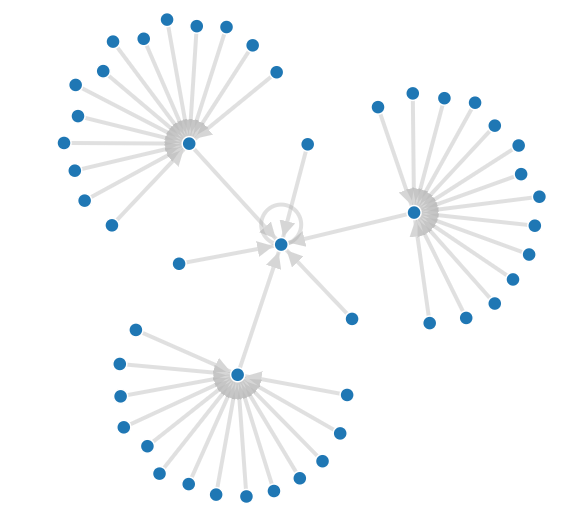}
	\end{center}
	\caption{Functional graph of $f(X)=(3X^7+4X)(X^7-X)$ over $\F_{49}$.}
\end{figure}
\end{Ex}

Since this describes the entire functional graph when $\delta_1=0$ or $\delta_2=0$, for the remainder of this section we  assume $\delta_1\delta_2\neq 0$. Hence, $f^{-1}(0,0)=\{\langle u,0\rangle\mid u\in\Fq\}$ and $\#f^{-1}(0,0)=q$. 
By $(\ref{f1})$,  $\langle u,0\rangle$ has nonempty preimage if and only if $\frac{u}{\delta_1}$ is a $n$th power in $\Fq$. In addition, the elements in its preimage are of the form $\langle 0,\lambda\rangle$, where $\lambda^n=\frac{u}{\delta_1}$, and they have no preimage. We point out that for every element $\langle 0,v\rangle \in\beta\Fq$, there is an element $\langle u,0\rangle\in\Fq$ such that $v^n=\frac{u}{\delta_1}$ and $\langle 0,v\rangle $ is directed to $\langle u,0\rangle$.

In conclusion, we have $\frac{q-1}{\mathfrak{g}(n)}$ elements that are $n$th powers in $\Fq$, with $\mathfrak{g}(n)$ elements in their preimage with no preimage, in a total of $q+\frac{q-1}{\mathfrak{g}(n)}\cdot \mathfrak{g}(n)=2q-1$ elements. Furthermore, this component contains all the elements of the set $\Fq\cup\beta\Fq$. This proves the following.

\begin{Th}
    The connected component that contains zero in the functional graph of $f$ is composed by a cycle of length one, with $q-1$ vertices directed to it and, among those, $\frac{q-1}{\mathfrak{g}(n)}$ have $\mathfrak{g}(n)$ elements in their preimage.
\end{Th}

\begin{Ex}
Let $q=13$, $n=4$ and $f(X)=(7X^{13}+3X)(X^{13}-X)^{3}$. Then the connected component that contains $0$ is isomorphic to
\begin{figure}[H]
	\begin{center}
		\includegraphics[scale=0.4]{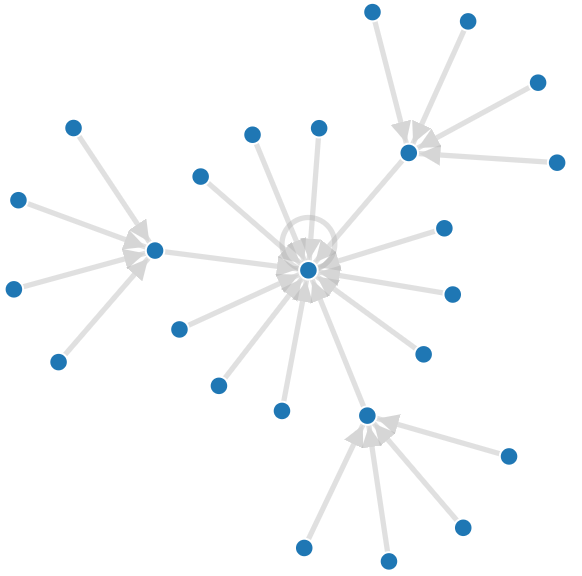}
	\end{center}
	\caption{Connected component that contains zero in the functional graph of $f(X)=(7X^{13}+3X)(X^{13}-X)^{3}$ over $\F_{169}$.}
\end{figure}
 
We notice that there are $2\cdot 13-1=25$ elements in this component.
\end{Ex}


\subsection{The cycles}

\begin{Df}

Given a function $f$ and a positive number $r$, we define $\mathcal{N}_f(r)$ to be the number of $r$-cycles in the functional graph of the map $f$. Since in this paper we analyse a fixed function, to simplify our notation we represent $\mathcal{N}_f(r)$ simply as $\mathcal{N}(r)$.
\end{Df}

Since $\langle0,0\rangle$ is always a fixed point, we know that $\mathcal{N}(1)\geq 1$, but it is not true that $\mathcal{N}(1)= 1$ for every choice of coefficients. Other fixed points may also appear in this functional graph as we see in the next example.

\begin{Ex}
Taking $q=13$, $f(X)=(10X^{13}+2X)(X^{13}-X)^{3}$ and $\beta\in\Fqd\setminus\Fq$ such that $\beta^2=11$, we obtain six copies of the following component that contains a fixed point 
\begin{figure}[H]
	\begin{center}
		\includegraphics[scale=0.4]{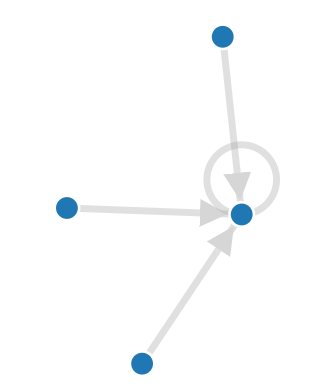}
	\end{center}
	\caption{Connected component that contains a fixed element in the functional graph of $f(X)=(10X^{13}+2X)(X^{13}-X)^{3}$ over $\F_{169}$.}
\end{figure}

\end{Ex}

Those fixed points different from $\langle0,0\rangle$ are counted along with the cycles of greater length. We also describe the pre-cycle of all the connected points further ahead.

\begin{Df}
For every $m\in\Z$ and $a\in\Fq$, we define the function $\chi_m(a)$ as $$\chi_m(a)=\begin{cases}0,&\ \text{if $a=0$}\\ 1,&\ \text{if $a$ is a $m$th power in $\Fq^*$}\\ -1,&\ \text{if $a$ is not a $m$th power in $\Fq^*$}. \end{cases}$$
In the case when $m=2$, $\chi_2$ is known as the quadratic character. 
\end{Df}

\begin{Th}\label{OddCycle}
There are odd cycles different from the $\langle0,0\rangle$ cycle in the functional graph of $f$ if and only if $\chi_2(a^2-c^2)=-1$. 

If $\chi_2(a^2-c^2)=-1$ and $\gamma$ is such that $\gamma^2=\frac{(a-c)b}{(a+c)}$, then, for $j\geq0$, the number of cycles of length $2j+1$, that are different from the cycle of $\langle0,0\rangle$, is $$\mathcal{N}(2j+1)=\frac{\Delta_o}{\mathfrak{g}(n-1) (2j+1)}\cdot\sum_{i\mid (2j+1)}\mu\left(\frac{2j+1}{i}\right)\mathfrak{g}(n^i-1),$$ where 
 
 $$\Delta_o:=\begin{cases}0&\text{if $\chi_{\mathfrak{g}(n-1)}(\gamma\delta_2)=\chi_{\mathfrak{g}(n-1)}(-\gamma\delta_2)=-1$,}\\ 
 2& \text{if $\chi_{\mathfrak{g}(n-1)}(\gamma\delta_2)=\chi_{\mathfrak{g}(n-1)}(-\gamma\delta_2)=1$,}\\
 1& \text{otherwise.}
 \end{cases}
 $$

\end{Th}
\begin{proof}
    We have proved that if $X=x+y\beta\in\Fqd$ and $x=0$ or $y=0$, then $X$ is in the connected component of $0$. 
    Therefore, suppose $xy\neq 0$. Recalling $(\ref{f1})$, we have that $f(x,y)=\langle\delta_1y^n,\delta_2xy^{n-1}\rangle.$ Applying $f$ again, it follows that
    \begin{align}
    f^{(2)}(x,y)&=\langle \delta_1(\delta_2xy^{n-1})^n,\delta_2\delta_1y^n(\delta_2xy^{n-1})^{n-1}\rangle \nonumber\\
    &=\delta_1\delta_2^nx^{n-1}y^{n^2-n}\langle x,y\rangle .\label{f2v}
\end{align}
    Writing $g(x,y)=\delta_1\delta_2^nx^{n-1}y^{n^2-n}$, for all $v\in\Fq$, we get \begin{align*}
        f^{(2)}(vx,vy)&=g(vx,vy)\langle vx,vy\rangle \nonumber\\
        &=v^{n^2}g(x,y)\langle x,y\rangle .
    \end{align*}
Furthermore, taking $v=g(x,y)$, we have\begin{align*}
        f^{(4)}(x,y)&=f^{(2)}(g(x,y)\langle x,y\rangle )\\&=g(x,y)^{n^2+1}\langle x,y\rangle . \end{align*}
Therefore, by induction, we conclude that \begin{align}
    f^{(2j)}(x,y)&=g(x,y)^{1+n^2+n^4+\D+n^{2j-2}}\langle x,y\rangle \nonumber\\
    &=g(x,y)^{\frac{n^{2j}-1}{n^2-1}}\langle x,y\rangle \label{f2nv}.
\end{align}
Now, applying $f$ once more  \begin{align}
    f^{(2j+1)}(x,y)&=f(g(x,y)^{\frac{n^{2j}-1}{n^2-1}}\langle x,y\rangle )\nonumber\\
    &=g(x,y)^{\frac{n^{2j+1}-n}{n^2-1}}y^{n-1}\langle \delta_1y,\delta_2x\rangle.
\end{align}
If there is a cycle of odd length, then we must have \begin{equation}
    g(x,y)^{\frac{n^{2j+1}-n}{n^2-1}}y^{n-1}\langle \delta_1y,\delta_2x\rangle =\langle x,y\rangle,\label{oddfix}
\end{equation} for some $\langle x,y\rangle \in(\Fq^*)^2$, which implies that  $\frac{\delta_1y}{x}=\frac{\delta_2x}{y}$ and $\frac{\delta_1}{\delta_2}=(\frac{x}{y})^2$.

Since $\frac{\delta_1}{\delta_2}=\frac{(a-c)b}{a+c}$, in order to obtain a solution, we must have $\chi_2(a^2-c^2)=-1$. Let $\gamma\in\Fq$ be such that $\gamma^2=\frac{\delta_1}{\delta_2}$. Then $x=\gamma y$, and, applying this relation in the first coordinate of $(\ref{oddfix})$, we conclude that 
\begin{align*}
    g(x,y)^{\frac{n^{2j+1}-n}{n^2-1}}y^{n-1}\delta_1y&=\gamma y\\
    (\delta_1\delta_2^nx^{n-1}y^{n^2-n})^{\frac{n^{2j+1}-n}{n^2-1}}y^{n-1}\gamma\delta_2&=1\\
    ((\gamma^2\delta_2)\delta_2^{n}(\gamma y)^{n-1}y^{n^2-n})^{\frac{n^{2j+1}-n}{n^2-1}}y^{n-1}\gamma\delta_2&=1,\\
\end{align*}
therefore
$$y^{n^{2j+1}-1}=((\gamma\delta_2)^{-1})^{\frac{n^{2j+1}-1}{n-1}}.$$
There exists $y$ that satisfies such equation if and only if $\gamma\delta_2$ is a $\mathfrak{g}(n-1)$th power in $\Fq$. In that case, for each $j$ we have $\mathfrak{g}(n^{2j+1}-1)$ solutions. Furthermore, $-\gamma$ also satisfies $(-\gamma)^2=\frac{\delta_1}{\delta_2}$, then if $-\gamma\delta_2$ is a $\mathfrak{g}(n-1)$th power in $\Fq$, we have $\mathfrak{g}(n^{2j+1}-1)$ more solutions to $(\ref{oddfix})$. Defining $\Delta_o$ as in the statement of the theorem, we get a total of $\Delta_o\cdot\mathfrak{g}(n^{2j+1}-1)$ solutions.

We notice that if $\langle x,y\rangle$ is an element that satisfies $$g(x,y)^{\frac{n^{i}-n}{n^2-1}}y^{n-1}\langle \delta_1y,\delta_2x\rangle =\langle x,y\rangle$$ for some $i$ that divides $2j+1$, then it also satisfies (\ref{oddfix}). Thus, they appear among the $\Delta_o\cdot\mathfrak{g}(n^{2j+1}-1)$ solutions, although they are not elements of a $(2j+1)$-cycle. Therefore, $$\Delta_o\cdot\frac{\mathfrak{g}(n^{2j+1}-1)}{\mathfrak{g}(n-1)}=\sum_{i\mid 2j+1}i\mathcal{N}(i).$$
Hence, by the M\"obius inversion formula, we conclude that $$(2j+1)\mathcal{N}(2j+1)=\frac{\Delta_o}{\mathfrak{g}(n-1)}\cdot\sum_{i\mid (2j+1)}\mu\left(\frac{2j+1}{i}\right)\mathfrak{g}(n^i-1)$$ and this concludes our proof.

\end{proof}

The following result is a direct consequence of this theorem.

\begin{Co}
The number of fixed points is always $\Delta_o+1$.
\end{Co}

Now that we have the number of cycles of odd length, we must calculate the number of cycles of even length.

\begin{Th}\label{evenlength}
Let $\alpha\in\Fq^*$ be a generator and $q-1=sr$, where $r$ is the maximum such that $\gcd(r,n)=1$. Let $l_0$ be the smallest nonegative integer such that $\delta_1\delta_2=\alpha^{l_0}$. For every $0\leq j< \frac{r}{\mathfrak{g}(n-1)}$, define  $$d_j=\frac{q-1}{s\gcd(\frac{q-1}{s},t_0+j\mathfrak{g}(n-1))}$$ and $k_j=\ord_{(n^2-1)d_j}(n^2)$, where $t_0$ is the minimal integer that satisfies $1\leq t_0\leq r$ and $t_0\equiv s^{-1}l_0\pmod {\mathfrak g(n-1)}$.  Then, \begin{center}
    $\mathcal{N}(2k_j)=\frac{\varphi(d_j)(q-1)\mathfrak{g}(n-1)-\Delta_e}{2k_j}$,
\end{center} where $$\Delta_e=\begin{cases}0,& \text{if $k_j$ is even}\\ k_j\cdot\mathcal{N}(k_j),& \text{if $k_j$ is odd.}\end{cases}$$ and those are the only cycles of even length.

\end{Th}
\begin{proof}
    Suppose $X=\langle x,y\rangle $ is an element of a cycle of length $2k>0$. This implies that $k$  is the minimal such that
    \begin{equation}
        (\delta_1\delta_2^nx^{n-1}y^{n^2-n})^{\frac{n^{2k}-1}{n^2-1}}=1.\label{eveneq}
    \end{equation}    
   If $\delta_1\delta_2^nx^{n-1}y^{n^2-n}=\theta\in\Fq^*$, then  $\frac{n^{2k}-1}{n^2-1}\equiv0\pmod{\ord(\theta)}$, which is equivalent to $n^{2k}\equiv1\pmod{(n^{2}-1)\ord(\theta)}$. However, there exists such $k$ if and only if $\gcd(\ord(\theta),n)=1$. Then, for every $2k$-cycle, there is a $d$, divisor of $q-1$ with $\gcd(d,n)=1$, such that $k=\ord_{(n^{2}-1)d}(n^2)$.
    
  
    Now, let $\theta$ be an element of order $d$, with $\gcd(d,n)=1$. Taking $\alpha$ a generator of the multiplicative group $\Fq^*$, we have $\theta=\alpha^i$ for some $i$ and 
    $$d=\ord(\alpha^{i})=\frac{q-1}{\mathfrak{g}(i)}.$$
    Since $\gcd(d,n)=1$, $s$ must divide $i$, hence $i=st$ for some $1\leq t\leq s$. 
    
    The equation $\delta_1\delta_2^nx^{n-1}y^{n^2-n}=\alpha^{st}$ is only solvable when $t$ is such that $\frac{\alpha^{st}}{\delta_1\delta_2}$ is a $\mathfrak{g}(n-1)$th power in $\Fq$. In other words, writing $\delta_1\delta_2=\alpha^{l_0}$ and $\frac{\alpha^{st}}{\delta_1\delta_2}=\alpha^{st-l_0}$, since $\gcd(s,n-1)=1$, we must have $$t\equiv s^{-1}l_0\pmod{\mathfrak{g}(n-1)}.$$ If $t_0$ is the minimum integer that satisfies $1\leq t_0\leq r$ and $t_0\equiv s^{-1}l_0\pmod{\mathfrak{g}(n-1)}$, then $t=t_0+j\mathfrak{g}(n-1)$, for some $0\leq j\leq\frac{r-t_0}{\mathfrak{g}(n-1)}$. Thus $$d=d_j=\frac{q-1}{s\gcd(\frac {q-1}s, t_0+j\mathfrak{g}(n-1))}$$
    For each  $0\leq j\leq\frac{r-t_0}{\mathfrak{g}(n-1)}$, any element $x+y\beta\in\Fqd$ with $\delta_1\delta_2^nx^{n-1}y^{n^2-n}=\theta$ and $\ord(\theta)=d_j$ satisfies (\ref{eveneq}) taking $k=k_j:=\ord_{(n^2-1)d_j}(n^2)$. We highlight that this does not imply necessarily that $x+y\beta$ is an element of a $2k_j$-cycle, as it could be an element of a $m$-cycle, where $m$ divides $2k_j$. Suppose that is the case. If $m$ is even, then $m=2m'$ and $$(\delta_1\delta_2^nx^{n-1}y^{n^2-n})^{\frac{n^{2m'}-1}{n^2-1}}=1.$$ However, this means that $k_j=\ord_{(n^2-1)d_j}(n^2)$ divides $m'$ and $k_j=m'$. Similarly, if $m$ is odd, then $m$ divides $k_j$ and  $$(\delta_1\delta_2^nx^{n-1}y^{n^2-n})^{\frac{n^{2m}-1}{n^2-1}}=1.$$ Since $k_j=\ord_{(n^2-1)d_j}(n^2)$, then $k_j$ divides $m$, which imply that $k_j=m$ and $k_j$ is odd. 
    
    To summarize, taking $x+y\beta$ such that $\delta_1\delta_2^nx^{n-1}y^{n^2-n}=\theta$ and $\ord(\theta)=d_j$, if $k_j$ is even then $x+y\beta$ is an element of a $2k_j$-cycle, if $k_j$ is odd then $x+y\beta$ is an element of either a $2k_j$-cycle or a $k_j$-cycle.

    Fixing $d_j$, in order to have $\ord(\alpha^{t})=d_j$, we must choose $t$ such that $\mathfrak{g}(t)=\frac{q-1}{d_j}$, which is equivalent to taking $t=\frac{q-1}{d_j}t'$ 
    with $\gcd(t',d_j)=1$. Since $1\leq t\leq q-1$, then $1\leq t'\leq d_j$. Hence there are $\varphi(d_j)$ choices for $t$.
    
    Furthermore, for each $t$, there are $(q-1)\mathfrak{g}(n-1)$ ordered pairs $\langle x,y\rangle $ such that $\delta_1\delta_2^nx^{n-1}y^{n^2-n}=\alpha^{t}$. Consequently, if $k_j$ is even, then the number of elements in cycles of length $2k_j$ is $$\varphi(d_j)(q-1)\mathfrak{g}(n-1).$$ If $k_j$ is odd, then we must subtract the elements that are in $k_j$-cycles, hence, there are \begin{center}
        $\varphi(d_j)(q-1)\mathfrak{g}(n-1)-k_j\,\mathcal{N}(k_j),$
    \end{center} where $\mathcal{N}(k_j)$ is given by the formula presented in Theorem \ref{OddCycle}.
    
    Since there are $2k_j$ elements in each cycle, there are $\frac{\varphi(d_j)(q-1)\mathfrak{g}(n-1)-\Delta_e}{2k_j}$ cycles, where $$\Delta_e=\begin{cases}0,& \text{if $k_j$ is even}\\ k_j\,\mathcal{N}(k_j),& \text{if $k_j$ is odd.}\end{cases}$$
\end{proof}

\subsection{The trees}
With the number and length of the cycles, we now study the shape of the pre-cycles. Each pre-cycle is composed by non-periodic elements that are directed to a single periodic element after a number of iterations of $f$, therefore, forming tree like structures hanging from the cyclic element.
\begin{Df}
    For any integer $t$, we denote by $\zeta_t$ a (not necessarily primitive) $t$th root of unity in $\Fq$, that is, $(\zeta_t)^t=1$ in $\Fq$.
\end{Df}
\begin{Le}\label{gamma}
 If $a\in\Fqd^*$ is such that $a=\langle a_1,a_2\rangle$, then $\#f^{-1}(a_1,a_2)\in \{0,\mathfrak{g}(n)\}$.
\end{Le}
\begin{proof}
    If $f^{-1}(a_1,a_2)$ is non-empty and $\langle x,y\rangle\in f^{-1}(a_1,a_2)$, then it satisfies $\delta_1y^n=a_1$ and $\delta_2xy^{n-1}=a_2$. Let $\zeta_{n}\in\Fq$ be a non trivial $n$th root of unity, then $\langle x\zeta_{n},y\zeta_{n}\rangle $ is also a solution to $\delta_1y^n=a_1$ and $\delta_2xy^{n-1}=a_2$. Since there are $\mathfrak{g}(n)$ roots, we have at least $\mathfrak{g}(n)$ elements in $f^{-1}(a_1,a_2)$.
    
    If $\langle x',y'\rangle$ is an element in $f^{-1}(a_1,a_2)$, then $\delta_1y^n=\delta_1(y')^n$ and $\delta_2xy^{n-1}=\delta_2x'y'^{n-1}$, which implies that $y^n=(y')^n$ and $(\frac{y}{y'})^n=1$. Thus, they must differ by a $n$th root of unity, as well as $x$ and $x'$. Therefore, there are exactly $\mathfrak{g}(n)$ solutions.
\end{proof}

With this proof we obtain a relation between the $\mathfrak{g}(n)$th roots of unity and the preimage of an element. In the next theorem, we show that, in fact, there is a relation between those roots and every element in the tree hanging from a cyclic element. 

\begin{Df}
 Given a tree $\T$, we define the distance between two vertices as the number of edges contained in the path that connects them. If $P$ is a vertex in $\T$, then $P$ is in the $i$th level of $\T$ if its distance to the root of $\T$ is $i$.
\end{Df}






\begin{Th}\label{precyc}
    Let $q-1=sr$, where $r$ is the greatest integer such that $gcd(r,n)=1$, and $e$ be the largest integer such that $\mathfrak{g}(n^e)<s$. Then, any nonzero periodic element is the root of a tree with height $e+1$. 
\end{Th}
\begin{proof}
    Let $q-1=sr$, where $r$ is the greatest integer such that $gcd(r,n)=1$. Consequently, we have that $\mathfrak{g}(n)$ divides $s$. 

    Let $\langle x_0,y_0\rangle\neq \langle0,0\rangle$ be a periodic element and, for all $i\geq 1$, let $\langle x_i,y_i\rangle$ be a periodic point such that $$f(x_i,y_i)=\langle x_{i-1},y_{i-1}\rangle.$$ 
    Since $n$ is even and greater than $2$, we have that there are nontrivial $n$th roots of unity. From Lemma \ref{gamma}, we get that any nonperiodical element is in $f^{-1}(x_{0},y_{0})$ if and only if it is of the form $\langle \zeta_{n} x_1,\zeta_{n} y_1\rangle $ for some nontrivial $n$th root of unity $\zeta_{n}$. 
   
   Suppose that $\langle x,y\rangle $ is an element in $f^{-1}(\zeta_{n} x_1, \zeta_{n} y_1)$. This means that $\delta_1y^n=\zeta_{n} x_1$ and $\delta_2xy^{n-1}=\zeta_{n} y_1$. Moreover, we have that $\delta_1y_2^n=x_1$ and $\delta_2x_2y_2^{n-1}=y_1$. Then, combining the four equations we obtain that $\delta_1y^n=\zeta_{n}\delta_1y_2^n$ and $\delta_2xy^{n-1}=\zeta_{n}\delta_2x_2y_2^{n-1}$, which implies that $y^n=\zeta_{n} y_2^n$ and $x=\frac{x_2}{y_2}y$. Therefore, $\langle \zeta_{n} x_1,\zeta_{n} y_1\rangle $ has nonempty preimage if and only if $\zeta_{n}$ is a $n$th power in $\Fq$, that is, if there are $n^2$th roots of unity in $\Fq$ that are not $n$th roots of unity.
   
   If $s=\mathfrak{g}(n)$, then, for any $k\geq 1$, the number of $n^k$th roots of unity is $\mathfrak{g}(n^k)=s.$ Therefore, every $n^k$th root of unity is also a $n$th root of unity. Hence $\zeta_{n}$ is not a $n$th power and  $\langle \zeta_{n} x_1,\zeta_{n} y_1\rangle $ has no preimage. Thus, every cyclic element has $\mathfrak{g}(n)-1$ nonperiodic elements directed to it, each with no preimage.
   

    Now, suppose that $s>\mathfrak{g}(n)$. Moreover, let $e$ be the greatest integer such that $\mathfrak{g}(n^e)<s$. As an induction hypothesis suppose that if $1\leq k<e$, then $\langle x,y\rangle$ is an element in the $k$th level of the tree with root $\langle x_0,y_0\rangle$ if and only if $\langle x,y\rangle=\langle \zeta_{n^k}x_{k},\zeta_{n^k}y_{k}\rangle$, where $\zeta_{n^k}\in\Fq$ is a $n^k$th root of unity that is not a $n^{k-1}$th root of unity.

    Let $\zeta_{n^{k+1}}$ be a $n^{k+1}$th root of unity that is not a $n^k$th root of unity. We have that $$f(\zeta_{n^{k+1}}x_{k+1},\zeta_{n^{k+1}}y_{k+1})=\langle\zeta_{n^{k+1}}^nx_{k},\zeta_{n^{k+1}}^ny_{k}\rangle.$$ Since $(\zeta_{n^{k+1}}^n)^{n^k}=1$ and $(\zeta_{n^{k+1}}^n)^{n^{k-1}}=\zeta_{n^{k+1}}^{n^k}\neq 1$, by the induction hypothesis, it follows that $\langle\zeta_{n^{k+1}}^nx_{k},\zeta_{n^{k+1}}^ny_{k}\rangle$ is an element in the $k$th level of the tree with root $\langle x_0,y_0\rangle$. Hence, $\langle \zeta_{n^{k+1}}x_{k+1},\zeta_{n^{k+1}}y_{k+1}\rangle$ is an element in the $(k+1)$th level. 

    Conversely, suppose that $$\langle x,y\rangle \in f^{-1}(\zeta_{n^{k}}x_{k},\zeta_{n^{k}}y_{k}).$$  Then $\delta_1y^n=\zeta_{n^{k}}x_{k}$ and, since $\delta_1y_{k+1}^n=x_{k}$ by definition, we have that $\delta_1y^n=\zeta_{n^{k}}\delta_1y_{k+1}^n$. It follows that $y^n=\zeta_{n^{k}}y_{k+1}^n$ and $y=\theta y_{k+1}$, where $\theta^n=\zeta_{n^{k}}$. Furthermore, we have that $\theta^{n^k}=\zeta_{n^{k}}^{n^{k-1}}\neq 1$ and $\theta^{n^{k+1}}=\zeta_{n^{k}}^{n^k}=1$. Hence, $\theta$ is a $n^{k+1}$th root of unity and not a $n^k$th root of unity. Therefore, if $\langle x,y\rangle$ is an element in the $k$th level of the tree with root $\langle x_0,y_0\rangle$, then $\langle x,y\rangle=\langle \zeta_{n^k}x_{k},\zeta_{n^k}y_{k}\rangle$, where $\zeta_{n^k}\in\Fq$ is a $n^k$th root of unity that is not a $n^{k-1}$th root of unity.

    Finally, for any $j> 1$ we have that $\mathfrak{g}(n^{e+j})=\mathfrak{g}(n^{e+1})=s$. Hence, any $n^{e+j}$th roots of unity are also $e^{e+1}$ roots of unity, and there are no elements in the $(e+j)$th level.


\end{proof}

We have that the elements in the $i$th level of the hanging tree are in a bijection with the $n^i$th roots of unity in $\Fq$ that are not $n^{(i-1)}$th roots of unity. The number of such roots of unity is $\mathfrak{g}(n^i)-\mathfrak{g}(n^{i-1})$. Hence, for each $i>0$, there are $\mathfrak{g}(n^i)-\mathfrak{g}(n^{i-1})$ elements in the level $i$ of the hanging tree of each nonzero cyclic element.

In general, there are multiple nonisomorphic trees with the same amount of vertices in each level. Therefore, we use the following to determine precisely the shape of the hanging trees of the functional graph of $f$.

\begin{Df}
    Let $n$ be a fixed integer with $\mathfrak{g}(n)>1$, and write $q-1=sr$, where $r$ is the greatest integer such that $\gcd(r,n)=1$. We define the tree $\T_n(1)$ as a single vertex with $\mathfrak{g}(n)-1$ vertices directed to it.
    
    Let $e$ be the greatest integer such that $\mathfrak{g}(n^e)<s$. For every $1\leq i\leq e$ we define the tree $\T_n(i+1)$ in the following way: First we label the elements in the $i$th level of $\T_n(i)$ from left to right in ascending order skipping the numbers that are multiples of $\frac{\mathfrak{g}(n^i)}{\mathfrak{g}(n^{i-1})}$. Then we attach $\mathfrak{g}(n)$ new vertices to each vertex labeled with a multiple of $\frac{\mathfrak{g}(n^{i})\mathfrak{g}(n)}{\mathfrak{g}(n^{i+1})}$.
\end{Df}

\begin{Th}\label{tree}
    Let $n$ be a fixed integer and $q-1=rs$, where $r$ is the greatest integer such that $\gcd(n,r)=1$. If $e$ is the largest integer such that $\mathfrak{g}(n^e)<s$, then each nonzero cyclic element is the root of a tree isomorphic to $\T_n(e+1)$ 
\end{Th}
\begin{proof}
    For any $1\leq i\leq e+1$, in the proof of Theorem \ref{precyc} we show that the elements in the $i$th level of the tree hanging from any cyclic element in the functional graph of $f$ correspond to the $n^i$th roots of unity that are not $n^{i-1}$th roots of unity in $\Fq$. Hence, in this proof we represent the elements of the tree by their corresponding roots of unity. Let $\zeta_{n^i}$ be one of such roots. We may assume $s>\mathfrak{g}(n)$ and $i>1$, since both cases were completely described in the proof of Theorem \ref{precyc}.
    
    Let $\alpha$ be a generator of the multiplicative group $\Fq^*$. We have that $\zeta_{n^i}=\alpha^{\frac{q-1}{\mathfrak{g}(n^i)}t}$, for some $1\leq t<\frac{q-1}{\mathfrak{g}(n^i)}$ and $t$ is not a multiple of $\frac{\mathfrak{g}(n^{i})}{\mathfrak{g}(n^{i-1})}$. Since $\alpha^{\frac{q-1}{\mathfrak{g}(n^i)}t}$ is an element in the $i$th level, there is an element $\zeta_{n^{i-1}}$ in the $(i-1)$th level such that $f(\zeta_{n^{i-1}})=\alpha^{\frac{q-1}{\mathfrak{g}(n^i)}t}$. 
    We want to find necessary and sufficient conditions for $\zeta_{n^{i-1}}$, in order for it to have a preimage.

    Writing $t=t_1+t_2\frac{\mathfrak{g}(n^i)}{\mathfrak{g}(n)}$, where $1\leq t_1<\frac{\mathfrak{g}(n^i)}{\mathfrak{g}(n)}$ and $0\leq t_2<\mathfrak{g}(n)$, we have that $$\zeta_{n^{i-1}}=\alpha^{\frac{q-1}{\mathfrak{g}(n^i)}(t_1+t_2\frac{\mathfrak{g}(n^i)}{\mathfrak{g}(n)})}n=(\alpha^{\frac{n}{\mathfrak{g}(n)}})^{\frac{q-1}{\mathfrak{g}(n^i)}(t_1\mathfrak{g}(n))}.$$ As $\mathfrak{g}(\frac{n}{\mathfrak{g}(n)})=1$, we have that $\alpha^{\frac{n}{\mathfrak{g}(n)}}=\alpha'$ is also a generator of $\Fq^*$. Moreover, $\zeta_{n^{i-1}}$ is a $n^{i-1}$th root of unity that is not a $n^{i-2}$th root of unity in $\Fq$. Hence, writing it as a power of $\alpha'$, we have $\zeta_{n^{i-1}}=(\alpha')^{\frac{q-1}{\mathfrak{g}(n^{i-1})}m}$, for some $1\leq m<\frac{q-1}{\mathfrak{g}(n^{i-1})}$ such that $m$ is not a multiple of $\frac{\mathfrak{g}(n^{i-1})}{\mathfrak{g}(n^{i-2})}$. Thus, $$(\alpha')^{\frac{q-1}{\mathfrak{g}(n^{i-1})}m}=(\alpha')^{\frac{q-1}{\mathfrak{g}(n^i)}(t_1\mathfrak{g}(n))}$$ and $(\alpha')^{\frac{q-1}{\mathfrak{g}(n^{i-1})}m}$ has a preimage if and only if $m$ is a multiple of $\frac{\mathfrak{g}(n^{i-1})\mathfrak{g}(n)}{\mathfrak{g}(n^{i})}$.

\end{proof}

\begin{Ex}
Let $q=13$, $n=2$ and $f(X)=(X^{13}+3X)(X^{13}-X)$. From the beginning of this section, we have seen that the connected component that contains zero is isomorphic to 
\begin{figure}[H]
	\begin{center}
		\includegraphics[scale=0.4]{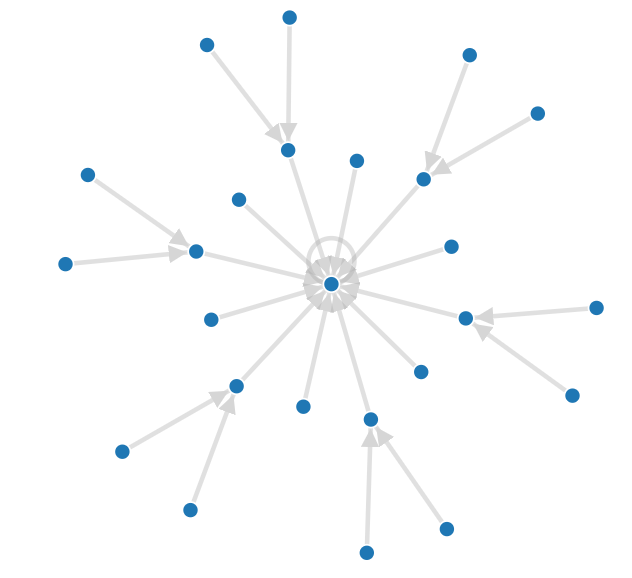}
	\end{center}
	\caption{Connected component that contains zero in the functional graph of $f(X)=(X^{13}+3X)(X^{13}-X)$ over $\F_{169}$.}
\end{figure}

We observe that $\mathfrak{g}(n)=2$ and $\mathfrak{g}(n^2)=4=s$, then $e=1$ and every nonzero cyclic element has a tree isomorphic to $\T_2(2)$ attached to it. Let us now find the odd cycles. Let $\beta\in\Fqd$ be such that $\beta^2=2$. Since $\chi_2(8)=-1$, there are odd cycles in this functional graph. Taking $\gamma=5$, we have $\gamma^2=\frac{(a-c)b}{(a+c)}$ and $\Delta_o=2$. Hence, $\mathcal{N}(1)=2$ and $\mathcal{N}(2j+1)=0$ for every $j>0$. Then there are $2$ copies of the following connected component.
\begin{figure}[H]
	\begin{center}
		\includegraphics[scale=0.4]{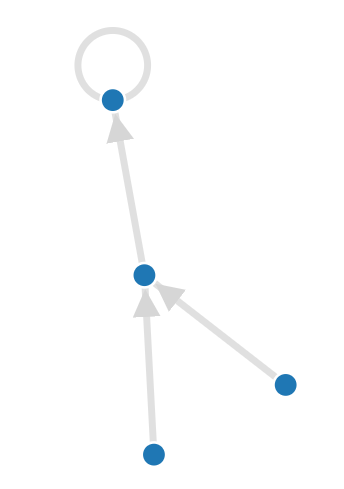}
	\end{center}
	\caption{Connected component containing cycle of length $1$ in the functional graph of $f(X)=(X^{13}+3X)(X^{13}-X)$ over $\F_{169}$.}
\end{figure}

For the even cycles, take $\alpha=2$. We have that $s=4$ and $r=3$. From simple calculations we obtain $\delta_1\delta_2=-1$. From Theorem $\ref{evenlength}$ we have $l_0=6$, $t_0=1$ and the possible values for $j$ are $\{0,1,2\}$. Thus, $k_0=k_1=3$ and $k_2=1$, which implies that $$\mathcal{N}(6)=\frac{2\cdot 12}{2\cdot 3}=4$$ and $$\mathcal{N}(2)=\frac{12-2}{2}=5.$$ Therefore, there are $4$ connected components isomorphic to
\begin{figure}[H]
	\begin{center}
		\includegraphics[scale=0.4]{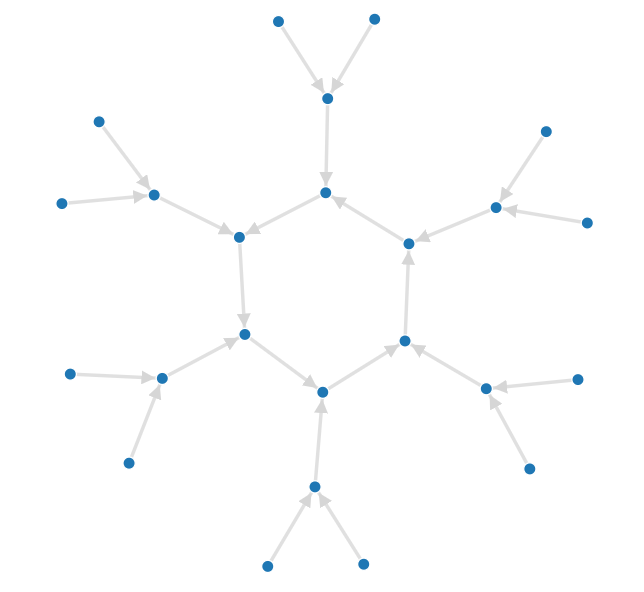}
	\end{center}
	\caption{Connected component containing cycle of length $6$ in the functional graph of $f(X)=(X^{13}+3X)(X^{13}-X)$ over $\F_{169}$.}
\end{figure}
and $5$ connected components isomorphic to 
\begin{figure}[H]
	\begin{center}
		\includegraphics[scale=0.3]{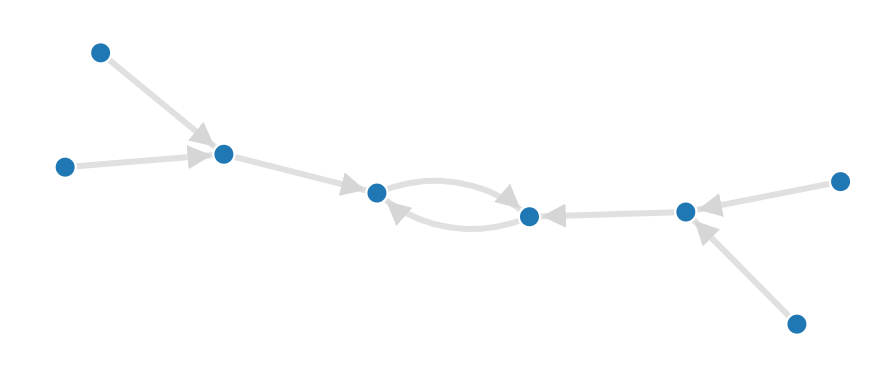}
	\end{center}
	\caption{Connected component containing cycle of length $2$ in the functional graph of $f(X)=(X^{13}+3X)(X^{13}-X)$ over $\F_{169}$.}
\end{figure}
\end{Ex}

\begin{Ex}
Now let $q=13$, $n=6$ and $f(X)=(X^{13}+3X)(X^{13}-X)^{5}$. The connected component that contains zero is isomorphic to 
\begin{figure}[H]
	\begin{center}
		\includegraphics[scale=0.5]{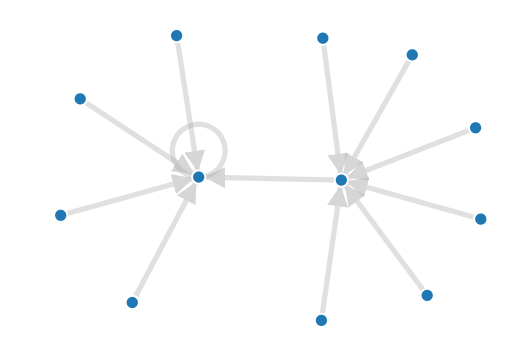}
	\end{center}
	\caption{Connected component containing zero in the functional graph of $f(X)=(X^{13}+3X)(X^{13}-X)^{5}$ over $\F_{169}$.}
\end{figure}

In this case, we have $n=6$, thus $\mathfrak{g}(n)=6$ and $\mathfrak{g}(n^2)=12=s$. Hence every tree is isomorphic to $\T_6(2)$.
Again, $\chi_2(a^2-c^2)=-1$ and taking $\gamma=-1$, we have $\Delta_0=2$ and we conclude that $\mathcal{N}(1)=2,$ and $\mathcal{N}(2j+1)=0$ for every other $j$. Therefore, there are $2$ copies of
\begin{figure}[H]
	\begin{center}
		\includegraphics[scale=0.5]{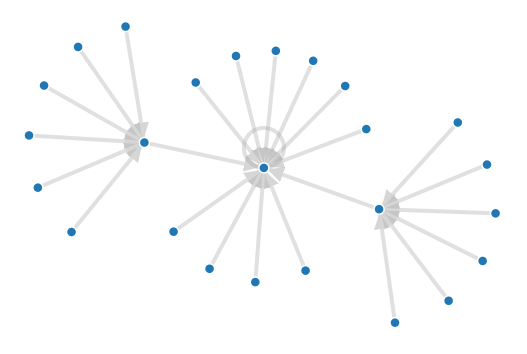}
	\end{center}
	\caption{Connected component containing fixed element in the functional graph of $f(X)=(X^{13}+3X)(X^{13}-X)^{5}$ over $\F_{169}$.}
\end{figure}

For the even cycles, take $\alpha=2$ a generator of $\Fq^*$, $s=12$ and $r=1$. Then $\delta_1\delta_2=-1$ and, from Theorem $\ref{evenlength}$, we have $l_0=6$, $t_0=1$, and the only possible value for $j$ is $0$. Then $k_{0}=1$ and $\mathcal{N}(2)=5$. Hence, there are $5$ copies of 
\begin{figure}[H]
	\begin{center}
		\includegraphics[scale=0.5]{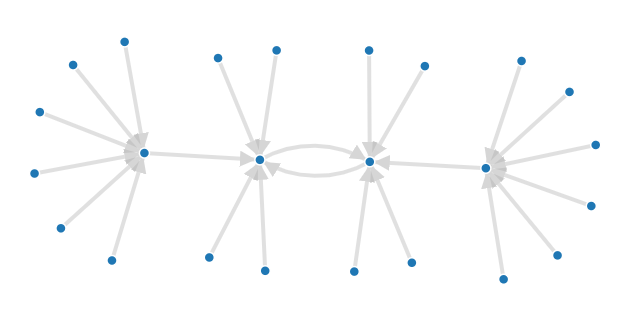}
	\end{center}
	\caption{Connected component containing a cycle of length $2$ in the functional graph of $f(X)=(X^{13}+3X)(X^{13}-X)^{5}$ over $\F_{169}$.}
\end{figure}
and those are the only elements in the functional graph of $f$.
\end{Ex}

\section{The dynamics for $n$ odd}

Let $a$ and $c$ be elements of $\Fq$. Now, we look at the second case, where $n$ is an odd integer. 

Again, we look at $f$ as a map of vector spaces over $\Fq\times \Fq$. Writing $X=x+y\beta$, we have 
\begin{align*}
    f(X)&=(cX^q+aX)(X^{q}-X)^{n-1}\\
    &=(cx-cy\beta+ax+ay\beta)(x-y\beta-x-y\beta)^{n-1}\\
    &=((c+a)x+(a-c)y\beta)(-2y\beta)^{n-1}
\end{align*}
Since $n-1$ is even, defining $\delta_1=(a+c)$ and $\delta_2=(a-c)$, we obtain the following 
\begin{align}
    \Fq\times\Fq&\overset{f}{\to} \Fq\times\Fq\nonumber\\
    \langle x,y\rangle&\mapsto(4\beta^2)^{\frac{n-1}{2}}y^{n-1}\langle\delta_1x,\delta_2y\rangle.\label{f2}
\end{align}

\subsection{The zero component}

Immediately from (\ref{f2}), we get that $\langle0,0\rangle$ is a fixed point. We start by analysing its preimage.

\begin{Th}
    If $\delta_2=0$ then the functional graph of $f$ contains one connected component. Moreover, it is composed by a cycle of length one with $2q-2$ vertices directed to it and, amongst those, $q-1$ have $q-1$ elements in its preimage.
\end{Th}
\begin{proof}
    If $\delta_2=0$, then $\delta_1\neq 0$ by definition. Moreover, we have that $\langle x,y\rangle\in f^{-1}(0,0)$ implies that $(4\beta^2)^{\frac{n-1}{2}}y^{n-1}\delta_1x=0$, thus $x=0$ or $y=0$. Consequently, $$f^{-1}(0,0)=\langle0,0\rangle\cup\{\langle x,0\rangle\mid x\in\Fq^*\}\cup\{\langle 0,y\rangle\mid y\in\Fq^*\}.$$ For every $y\in\Fq^*$, $\langle 0,y\rangle$ has no preimage. On the other hand, for each $x\in\Fq^*$ there are $q-1$ elements in $f^{-1}(x,0)$. Thus there are $1+2(q-1)+(q-1)^2=q^2$ elements in this component and it is the only component.
\end{proof}

\begin{Ex}
Taking $q=7$ and $f(X)=(2X^{7}+2X)(X^{7}-X)^{2}$, we have $\delta_2=2-2=0$ and the functional graph is isomorphic to  
\begin{figure}[H]
	\begin{center}
		\includegraphics[scale=0.4]{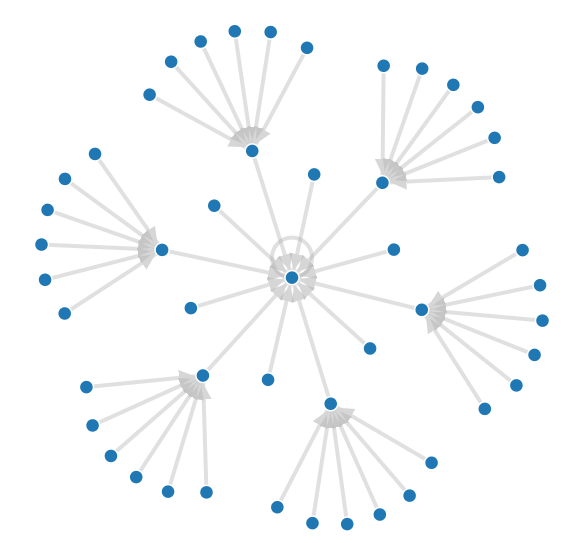}
	\end{center}
	\caption{Functional graph of $f(X)=(2X^{7}+2X)(X^{7}-X)^{2}$ over $\F_{49}$.}
\end{figure}
We point out that it contains all the $7^2$ elements of the field.
\end{Ex}

From now on we may assume $\delta_2\neq0$.
In that case, $f^{-1}(0,0)=\{\langle x,0\rangle\mid x\in\Fq\}$ and, for every $x\in\Fq^*$, $\langle x,0\rangle$ has no preimage. Therefore, the connected component of $\langle0,0\rangle$ contains $q$ elements, all directed to $\langle0,0\rangle$. 

\subsection{The cycles}

After obtaining the description of the component that contains zero, we study the remaining cycles. We start with cyclesof length $1$. 

\begin{Th}
 The number of fixed points of $f$ is $$\mathcal{N}(1)=\mathcal{N}_0(1)+\mathcal{N}_*(1)+1,$$ where $$\mathcal{N}_0(1)=\begin{cases}
    \mathfrak{g}(n-1) &\text{if $\chi_{\frac{n-1}{2}}(\delta_2)=1$ and $\chi_{n-1}(\delta_2)=-1$}\\
    0 & \text{otherwise},
\end{cases}$$ is the number of nonzero fixed points with zero in the first coordinate
 and $$\mathcal{N}_*(1)=\begin{cases}
    (q-1)\mathfrak{g}(n-1) &\text{if $\chi_{\frac{n-1}{2}}(\delta_2)=1$, $\chi_{n-1}(\delta_2)=-1$ and $\delta_1=\delta_2$}\\
    0 & \text{otherwise},
    \end{cases}$$ is the number of fixed points with the first coordinate different than zero.
\end{Th}
\begin{proof}
Let $\langle x,y\rangle\in\Fq^2$ be an element such that $f(x,y)=\langle x,y\rangle$. This implies that $x$ and $y$ satisfy 
\begin{equation}
    x=(4\beta^2)^{\frac{n-1}{2}}y^{n-1}\delta_1x\label{x}
\end{equation} 
and 
\begin{equation}
    y=(4\beta^2)^{\frac{n-1}{2}}y^{n-1}\delta_2y\label{y}.
\end{equation}
We have already seen that $\langle0,0\rangle$ satisfies the equations, as it is always a fixed point. Moreover, every element $\langle x,0\rangle$ is directed to $\langle 0,0\rangle$, thus, we suppose $y\neq 0$.  

If $x=0$, then $y^{n-1}=\frac{1}{(4\beta^2)^{\frac{n-1}{2}}\delta_2}$. We notice that $(4\beta^2)^{\frac{n-1}{2}}$ is a $\frac{n-1}{2}$ power but not a $n-1$ power, hence we only have such $y$ when $\delta_2$ is a $\frac{n-1}{2}$ power but not a $n-1$ power. In that case we have $\mathfrak{g}(n-1)+1$ fixed points of the form $\langle 0,y\rangle$, including $\langle 0,0\rangle$.

If $x\neq 0$, then (\ref{x}) and (\ref{y}) are only solvable in $\Fq^*$ if $\delta_1=\delta_2$. Moreover, we must have that $\delta_2$ is a $\frac{n-1}{2}$ power and not a $n-1$ power in $\Fq^*$. In this case there are $(q-1)\mathfrak{g}(n-1)$ elements that satisfy both equations.
\end{proof}

To analyse the remaining cycles, we iterate the function $f$ multiple times and observe its shape. 

\begin{Df}
    For every $j\in\N$, we define the function $\mathfrak{u}(j)=\frac{n^j-1}{n-1}$.
\end{Df}

From $(\ref{f2})$ we know that $f(x,y)=(-2y\beta)^{n-1}\langle\delta_1x,\delta_2y\rangle$. Applying $f$ again we obtain \begin{align*}
    f^{(2)}(x,y)=&((-2y\beta)^{n}\delta_2)^{n-1}\langle (-2y\beta)^{n-1}\delta_1^2x,(-2y\beta)^{n-1}\delta_2^2y\rangle\\ 
    =&((-2y\beta)^{n+1}\delta_2)^{n-1}\langle\delta_1^2x,\delta_2^2y\rangle\\
    =&(-2y\beta)^{n^2-1}\delta_2^{n-1}\langle\delta_1^2x,\delta_2^2y\rangle.
\end{align*}
By induction, one can see that \begin{align*}
    f^{(j)}(x,y)=&(-2y\beta)^{n^{j}-1}\delta_2^{n^{j-1}+n^{j-2}+\cdots+n-j+1}\langle\delta_1^jx,\delta_2^jy\rangle\\
    =&((-2y\beta)^{n-1}\delta_2)^{\mathfrak{u}(j)}\langle\big(\frac{\delta_1}{\delta_2}\big)^jx,y\rangle.
\end{align*}
Again, since $n-1$ is even, we may rewrite it as 

$$f^{(j)}(x,y)=((4\beta^2)^{\frac{n-1}{2}}\delta_2y^{n-1})^{\mathfrak{u}(j)}\langle\big(\frac{\delta_1}{\delta_2}\big)^jx,y\rangle.$$

We would like to obtain conditions on the possible length of cycles of $f$. For that we need some number theoretical properties.

\begin{Df}
For each prime $p$ and integer $a$ we define $\nup(a)$ as the greatest integer such that $p^{\nup(a)}$ divides $a$.
\end{Df}

\begin{Le}[Lifting The Exponent Lemma \cite{Manea}
]\label{Lemaj}
Let $n$ be an odd integer. \begin{enumerate}
    \item If $p$ is an odd prime and $\nup(n-1)>0$, then $\nup(\mathfrak{u}(j))=\nup(j)$.
    \item If $p=2$ and $j$ is odd, then $\nup(\mathfrak{u}(j))=\nup(j)=0$.
    \item If $p=2$, $j$ is even and $n\equiv 1\pmod 4$, then $\nup(\mathfrak{u}(j))=\nup(j)$.
    \item If $p=2$, $j$ is even and $n\equiv 3\pmod 4$, then $\nup(\mathfrak{u}(j))=\nup(\frac{n+1}{2})+\nup(j)$.
\end{enumerate}     
\end{Le}

Now we prove the following theorem.

\begin{Th}\label{Theoj}
    Let $j$ be an integer such that $f^{(j)}(x,y)=\langle x,y\rangle$, for some element with $y\neq0$. Let $\alpha$ be a generator of the multiplicative group $\Fq^*$ and let $l_0$ be the smallest nonegative integer such that $(4\beta^2)^{\frac{n-1}{2}}\delta_2=\alpha^{l_0}$. If $\delta_1=0$, then $j$ is a multiple of $\tau_0$, where   $$\tau_0=2^{\epsilon_0}\cdot\prod_{\overset{p\mid\frac{n-1}{\gcd(n-1,l_0)}}{p\neq2}}p^{\max\{0,\nup(q-1)-\nup(l_0)\}},$$ and $$\epsilon_0=\begin{cases}
    0& \text{if $\nu_2(l_0)>0$}\\
    \nu_2(q-1)-\nu_2(\frac{n+1}{2})& \text{if $\nu_2(l_0)=0$}.
\end{cases}$$ If $\gcd(n-1,l_0)=n-1$, we define $\tau_0=1$.

\end{Th}

\begin{proof}
    
If $\delta_1=0$ then $f^{(j)}(x,y)=\langle x,y\rangle$ if and only if  $((4\beta^2)^{\frac{n-1}{2}}\delta_2y^{n-1})^{\mathfrak{u}(j)}=1$ and $x=0$. Therefore, let $\zeta_{\mathfrak{u}(j)}\in\Fq^*$ be a $\mathfrak{u}(j)$th root of unity, then \begin{equation}
    (4\beta^2)^{\frac{n-1}{2}}\delta_2y^{n-1}=\zeta_{\mathfrak{u}(j)}.\label{eq1}
\end{equation} and we have that $y^{n-1}=\frac{\zeta_{\mathfrak{u}(j)}}{(4\beta^2)^{\frac{n-1}{2}}\delta_2}$. We notice that it only has a solution if $\frac{\zeta_{\mathfrak{u}(j)}}{(4\beta^2)^{\frac{n-1}{2}}\delta_2}$ is a $n-1$ power. 


Let $\alpha$ be a generator of $\Fq^*$ and  $(4\beta^2)^{\frac{n-1}{2}}\delta_2=\alpha^{l_0}$. We have that $\zeta_{\mathfrak{u}(j)}=\alpha^{\frac{q-1}{\mathfrak{g}(\mathfrak{u}(j))}t}$, for some $t\in \N$.   Then, we can find $j$ and $t$ such that
$$\mathfrak{g}(n-1)\ \text{divides}\  \frac{q-1}{\mathfrak{g}(\mathfrak{u}(j))}t-l_0$$
if and only if
\begin{equation}
    \gcd \left(\mathfrak{g}(n-1),\frac{q-1}{\mathfrak{g}(\mathfrak{u}(j))}\right)
=\gcd \left(n-1,\frac{q-1}{\mathfrak{g}(\mathfrak{u}(j))}\right)\quad\text{divides}\quad l_0.\label{Cond}
\end{equation}
Writing $q-1=rs$, where $r$ is the maximum such that $\gcd(r,n)=1$, we can find $j_0$ such that $n^{j_0}\equiv 1 \pmod{r(n-1)}$. Therefore  $\mathfrak{g}(\mathfrak{u}(j_0))=r$ and 
$\frac{q-1}{\mathfrak{g}(\mathfrak{u}(j_0))}=s$. As $\gcd(s,n-1)=1$, we conclude that there exists at least one pair $(j,t)$ such that $\mathfrak{g}(n-1)$ divides $\frac{q-1}{\mathfrak{g}(\mathfrak{u}(j))}t+l_0.$

For (\ref{Cond}) to be satisfied, for every $p$ prime we must have $$\min\{\nup(n-1),\nup(q-1)-\min\{\nup(\mathfrak{u}(j)),\nup(q-1)\}\}\leq \nup(l_0),$$ which is equivalent to \begin{equation}
    \min\{\nup(n-1),\max\{\nup(q-1)-\nup(\mathfrak{u}(j)),0\}\}\leq \nup(l_0).\label{Cond2}
\end{equation}
If $\nup(n-1)\leq\nup(l_0)$ then (\ref{Cond2}) holds. Thus, we look at $p$ such that $\nup(n-1)>\nup(l_0)$. Then, we need that \begin{equation}
    \max\{\nup(q-1)-\nup(\mathfrak{u}(j)),0\}\leq\nup(l_0)\label{Cond3}.
\end{equation}
Again, if $\nup(q-1)\leq\nup(l_0)$, then (\ref{Cond3}) is satisfied. On the other hand, if $p$ is such that $\nup(q-1)>\nup(l_0)$ we have to find the values of $j$ such that $\nup(\mathfrak{u}(j))\geq\nup(q-1)-\nup(l_0)$.

Using Lemma \ref{Lemaj}, if $\nu_2(l_0)>0$, then $\nu_2(n-1)\leq\nu_2(l_0)$ if $n\equiv3\pmod4$ or $\nu_2(\mathfrak{u}(j))=\nu_2(j)$ if $n\equiv1\pmod4$. Hence, to satisfy (\ref{Cond3}) we must have $j$ divisible by $p^{\nup(q-1)-\nup(l_0)}$, for every prime $p$ that satisfies $\nup(n-1)>\nup(l_0)$ and $\nup(q-1)>\nup(l_0)$. 

If $\nu_2(l_0)=0$, then we must have $\nu_2(\mathfrak{u}(j))\geq\nu_2(q-1)$. From Lemma \ref{Lemaj}, we have that it is equivalent to $\nu_2(j)\geq \nu_2(q-1)-\nu_2(\frac{n+1}{2})$. Thus, we have that $j$ must be a multiple of $2^{\nu_2(q-1)-\nup(\frac{n+1}{2})}$ and $p^{\nup(q-1)-\nup(l_0)}$, for every prime $p\neq2$ that satisfies $\nup(n-1)>\nup(l_0)$ and $\nup(q-1)>\nup(l_0)$.

\end{proof}

For $\delta_1\neq 0$, we divide the  connected components into three categories: the one that contains zero and the elements of the form $\langle x,0\rangle$, which was already described; the ones that contain elements of the form $\langle0,y\rangle$; and the ones that contain the elements in which both coordinates are different from zero.

From $f(x,y)=(-2y\beta)^{n-1}\langle\delta_1x,\delta_2y\rangle$, one can see that the image and preimage of any element $\langle0,y\rangle$, with $y\neq 0$, is also be of the form $\langle0,y'\rangle$. Hence the second type has no intersection with the other ones. One can also see that the cycles of this type follow the same conditions as the case where $\delta_1=0$, therefore, their length must be a multiple of $\tau_0$.

On the other hand, in the third case, in addition to the condition above, we also need that $j$ satisfy $\big(\frac{\delta_1}{\delta_2}\big)^j=1$. Thus, $j$ is a multiple of $$\tau:=\lcm(\ord\big(\frac{\delta_1}{\delta_2}\big), \tau_0).$$ This proves the following theorem

\begin{Th}\label{tau}
Let $\langle x,y\rangle$ be a periodic element. If $x=y=0$, then its cycle length is $1$. If $xy\neq 0$ and $\delta_1\delta_2\neq 0$, then its cycle length is a multiple of $\tau$. In all other cases its cycle length is a multiple of $\tau_0$.
\end{Th}

Now the cycles of this functional graph are narrowed to $1$-cycles, $(t\tau_0)$-cycles and $(t\tau)$-cycles, where $t$ is any positive integer.

\begin{Th}
If $\delta_1= 0$ and $t$ is a positive integer such that $t\tau_0>1$, then there are $$\mathcal{N}(t\tau_0)=\frac{1}{t\tau_0}\sum_{j\mid t}\mu\left(\frac{t}{j}\right)(\mathfrak{g}(\mathfrak{u}(j\tau_0))\gcd\left(n-1,\frac{q-1}{\mathfrak{g}(\mathfrak{u}(j\tau_0))}\right)-\mathcal{N}(1)+1).$$ 
$(t\tau_0)$-cycles, and those are the only cycles of length greater than $1$.

If $\delta_1\neq0$ and $t$ is a positive integer such that $t\tau_0>1$, then there are
\begin{equation*} \mathcal{N}(t\tau_0)=\frac{1}{t\tau_0}\sum_{j\mid t}\mu\left(\frac{t}{j}\right)(\mathfrak{g}(\mathfrak{u}(j\tau_0))\gcd\left(n-1,\frac{q-1}{\mathfrak{g}(\mathfrak{u}(j\tau_0))}\right)-\mathcal{N}_0(1))\end{equation*}  $(t\tau_0)$-cycles of elements with zero in the first coordinate
and \begin{equation*} \mathcal{N}(t\tau)=\frac{1}{t\tau}\sum_{j\mid t}\mu\left(\frac{t}{j}\right)((q-1)\mathfrak{g}(\mathfrak{u}(j\tau))\gcd\left(n-1,\frac{q-1}{\mathfrak{g}(\mathfrak{u}(j\tau))}\right)-\mathcal{N}_*(1))\end{equation*} $(t\tau)$-cycles of elements with the first coordinate different from zero, and those are the only cycles of length greater than $1$.
\end{Th}
\begin{proof}
Let $\langle x,y\rangle$ be a nonfixed cyclic element. Thus, it must satisfy \begin{equation}
    f^{(k)}(x,y)=((-2y\beta)^{n-1}\delta_2)^{\mathfrak{u}(k)}\Big\langle\Big(\frac{\delta_1}{\delta_2}\Big)^{k}x,y \Big\rangle=\langle x,y\rangle \label{ftau}
\end{equation} for some $k>1$. By the previous theorem we have that $k=t\tau$, if $\delta_1\neq0$ and $x\neq0$, or $k=t\tau_0$, for some $t\in\Z^+$. 

As an initial case, let us suppose that $\delta_1=0$. Then, any cyclic element satisfies \begin{equation}
    f^{(k)}(x,y)=((-2y\beta)^{n-1}\delta_2)^{\mathfrak{u}(k)}\Big\langle0,y \Big\rangle=\langle x,y\rangle. \label{delta=0}
\end{equation} Consequently, we have that $x=0$ and $$((-2y\beta)^{n-1}\delta_2)^{\mathfrak{u}(k)}=1.$$ Hence  \begin{equation}
    (-2y\beta)^{n-1}\delta_2=\zeta_{\mathfrak{u}(k)}^i\label{yi}
\end{equation} for some $0\leq i< \mathfrak{g}(\mathfrak{u}(k))$.  From the proof of Theorem \ref{Theoj}, we have that (\ref{yi}) is solvable in $y$ only for $\frac{\mathfrak{g}(\mathfrak{u}(k))\gcd(n-1,\frac{q-1}{\mathfrak{g}(\mathfrak{u}(k))})}{\mathfrak{g}(n-1)}$ different $i$ modulo $\mathfrak{g}(\mathfrak{u}(k))$ and, for each $i$, there are $\mathfrak{g}(n-1)$ values for $y$ in $\Fq^*$ that solve $(\ref{yi})$. Therefore, there are $\mathfrak{g}(\mathfrak{u}(k))\gcd(n-1,\frac{q-1}{\mathfrak{g}(\mathfrak{u}(k))})$ nonzero elements that satisfy $(\ref{delta=0})$ for a given $k=t\tau_0$ when $\delta_1=0$.


If $t=1$, from Theorem \ref{tau} we conclude that the the only nonzero elements that satisfy $(\ref{delta=0})$ are the nonzero fixed points of $f$ and the points in $\tau_0$-cycles. Therefore the number of points in $\tau_0$-cycles is $\mathfrak{g}(\mathfrak{u}(\tau_0))\gcd(n-1,\frac{q-1}{\mathfrak{g}(\mathfrak{u}(\tau_0))})-\mathcal{N}(1)+1$ and $$\mathcal{N}(\tau_0)=\frac{1}{\tau_0}\cdot\left(\mathfrak{g}(\mathfrak{u}(k))\gcd(n-1,\frac{q-1}{\mathfrak{g}(\mathfrak{u}(k))})-\mathcal{N}(1)+1\right).$$
Now, let $t> 1$. Every point that satisfies $f^{(j\tau_0)}(x,y)=\langle x,y\rangle$, for $j$ a positive divisor of $t$, also satisfies $f^{(t\tau_0)}(x,y)=\langle x,y\rangle$. Consequently, the number of elements that satisfy $f^{(t\tau_0)}(x,y)=\langle x,y\rangle$ is the sum of all the elements that are in $j\tau_0$-cycles for every $j$ that divides $t$. In other words, $$\mathfrak{g}(\mathfrak{u}(t\tau_0))\gcd \left(n-1,\frac{q-1}{\mathfrak{g}(\mathfrak{u}(t\tau_0))}\right)-\mathcal{N}(1)+1=\sum_{j\mid t}j\tau_0 \mathcal{N}(j\tau_0).$$
By the M\"obius inversion formula, we conclude that \begin{equation} \mathcal{N}(t\tau_0)=\frac{1}{t\tau_0}\sum_{j\mid t}\mu\left(\frac{t}{j}\right)(\mathfrak{g}(\mathfrak{u}(j\tau_0))\gcd\left(n-1,\frac{q-1}{\mathfrak{g}(\mathfrak{u}(j\tau_0))}\right)-\mathcal{N}(1)+1).\label{cyclesdelta=0}\end{equation}

As seen before, if $\delta_1\neq0$, then we divide the cyclic elements in two sets: elements with $x=0$ and elements with $x\neq0$. We have that any periodic element with $x=0$ is contained in a cycle of length multiple of $\tau_0$ and we have a case analogous to when $\delta_1=0$. In this case, since we are fixing $x=0$, we must subtract only the nonzero fixed points with $x=0$. Therefore, the number of cycles of length $t\tau_0$ is \begin{equation} \mathcal{N}(t\tau_0)=\frac{1}{t\tau_0}\sum_{j\mid t}\mu\left(\frac{t}{j}\right)(\mathfrak{g}(\mathfrak{u}(j\tau_0))\gcd\left(n-1,\frac{q-1}{\mathfrak{g}(\mathfrak{u}(j\tau_0))}\right)-\mathcal{N}_0(1)).\end{equation}

When $x\neq 0$, the length of cycles is a multiple of $\tau$. Hence, we have $\mathfrak{g}(\mathfrak{u}(k))\gcd(n-1,\frac{q-1}{\mathfrak{g}(\mathfrak{u}(k))})$ choices for $y$ that satisfy $(\ref{delta=0})$ and, for each $y$, there are $q-1$ choices for $x$ different from zero that satisfy (\ref{ftau}). Thus, there are $(q-1)\mathfrak{g}(\mathfrak{u}(k))\gcd(n-1,\frac{q-1}{\mathfrak{g}(\mathfrak{u}(k))})$ elements that satisfy (\ref{ftau}). We now follow the same process as in the case where $\delta_1=0$, however, we must subtract the nonzero fixed elements with $x\neq0$. Thus, we conclude that \begin{equation*} \mathcal{N}(t\tau)=\frac{1}{t\tau}\sum_{j\mid t}\mu\left(\frac{t}{j}\right)((q-1)\mathfrak{g}(\mathfrak{u}(j\tau))\gcd\left(n-1,\frac{q-1}{\mathfrak{g}(\mathfrak{u}(j\tau))}\right)-\mathcal{N}_*(1)).\end{equation*}

\end{proof}

\subsection{The trees}

Finally, we investigate the behavior of the trees hanging from each cyclic element. Thus, we can find analogues to Lemma \ref{gamma} and  Theorem \ref{precyc}, in the case when $n$ is an odd integer. 

\begin{Le}\label{gammaodd}
 Let $\langle a_1,a_2\rangle$ a nonzero element in $\Fq\times\Fq$. If $\delta_1=0$, then $\#f^{-1}(a_1,a_2)\in \{0,q\cdot \mathfrak{g}(n)\}$. If $\delta_1\neq 0$, then $\#f^{-1}(a_1,a_2)\in \{0,\mathfrak{g}(n)\}$.
\end{Le}
\begin{proof}
    If $f^{-1}(a_1,a_2)$ is non-empty and $\langle x,y\rangle\in f^{-1}(a_1,a_2)$, then it satisfies \begin{equation}
        (4\beta^2)^{\frac{n-1}{2}}y^{n-1}\delta_1x=a_1\ \text{and}\ (4\beta^2)^{\frac{n-1}{2}}y^{n-1}\delta_2y=a_2. \label{eq}
    \end{equation} 

    
    If $\langle x',y'\rangle$ is an element in $f^{-1}(a_1,a_2)$, then it follows that $(-2y\beta)^{n-1}\delta_1x=(-2y'\beta)^{n-1}\delta_1x'$ and $(-2y\beta)^{n-1}\delta_2y=(-2y'\beta)^{n-1}\delta_2y'$. The last equation implies that $y^n=(y')^n$ and $(\frac{y}{y'})^n=1$. If $\mathfrak{g}(n)=1$, then we must have $y=y'$. If $\mathfrak{g}(n)>1$, then $y$ and $y'$ must differ by a $n$th root of unity. 
    
    If $\delta_1\neq0$, then we also conclude that $x=x'$, if $\mathfrak{g}(n)=1$, or $x$ and $x'$ differ by the same $n$th root of unity, if $\mathfrak{g}(n)>1$. Hence, there are exactly $\mathfrak{g}(n)$ solutions to (\ref{eq}). On the other hand, if $\delta_1=0$, then the last equation holds true for any value of $x'$. Therefore there are $q\cdot\mathfrak{g}(n)$ solutions to (\ref{eq}).
\end{proof}

\begin{Obs}
From this proof, we conclude that when $\delta_1=0$, any element with a nonempty preimage must be of the form $\langle 0,y\rangle$. 
\end{Obs}

We now prove an equivalent theorem to Theorem \ref{precyc}, for the case when $n$ is odd.

\begin{Th}\label{treeodd}
    Let $n$ be an odd integer and write $q-1=sr$, where $r$ is the greatest integer such that $gcd(r,n)=1$. If $\mathfrak{g}(n)>1$ we define $e$ as the largest integer such that $\mathfrak{g}(n^e)<s$. Then, any nonzero periodic element in the functional graph of $f$ is the root of a tree with height $e+1$. 

    If $\mathfrak{g}(n)=1$, then there are no trees hanging from nonzero periodic elements.
\end{Th}
\begin{proof}
    Let $\langle x_0,y_0\rangle$ be a nonzero periodic element and, for all $i\geq 1$, let $\langle x_i,y_i\rangle$ be a periodic point such that $$f(x_i,y_i)=\langle x_{i-1},y_{i-1}\rangle.$$ 
    If $\mathfrak{g}(n)=1$, then from Lemma \ref{gamma}, we get that the only element in $f^{-1}(x_{0},y_{0})$ is $\langle x_1,y_1\rangle$. Hence, there are no trees hanging from nonzero periodic elements.
    
    If $\mathfrak{g}(n)>1$, we have that there are nontrivial $n$th roots of unity. From Lemma \ref{gamma}, we get that any nonperiodic element is in $f^{-1}(x_0,y_0)$ if and only if it is of the form $\langle x_1',\zeta_{n} y_1\rangle $, where $x_1'=\zeta_{n}x_1$ if $\delta_1\neq0$, or $x_1'$ is any element in $\Fq$ otherwise.

   Suppose that $\langle x,y\rangle \in f^{-1}(\zeta_{n}x_1, \zeta_{n} y_1)$. From the remark made previously, we have that $x_1'=0$, if $\delta_1=0$. In any case, this means that \begin{equation*}
       (-2y\beta)^{n-1}\delta_2y=\zeta_{n} y_1.
   \end{equation*} Moreover, we have that 
   $(-2y_2\beta)^{n-1}\delta_2y_2=y_1$. Then, it follows that \begin{equation*}
       y^{n}=\zeta_{n}y_2^{n},
   \end{equation*}  which has a solution if and only if $\zeta_{n}$ is a $n$th power in $\Fq$. 
   Therefore, $\langle x_1',\zeta_{n} y_1\rangle $ has nonempty preimage if and only if $\zeta_{n}$ is a $n$th power in $\Fq$, that is, if there are $n^2$th roots of unity in $\Fq$ that are not $n$th roots of unity.
   
   If $s=\mathfrak{g}(n)$, then $e=0$ and, for any $k\geq 1$, the number of $n^k$th roots of unity is $\mathfrak{g}(n^k)=s.$ Therefore, every $n^k$th root of unity is also a $n$th root of unity. Hence $\zeta_{n}$ is not a $n$th power and  $\langle \zeta_{n} x_1,\zeta_{n} y_1\rangle $ has no preimage. Thus, every cyclic element is a root of a tree of hight $1$.
   

    Now, suppose that $s>\mathfrak{g}(n)$. As an induction hypothesis suppose that if $1\leq k<e$, then $\langle x,y\rangle$ is an element in the $k$th level of the tree with root $\langle x_0,y_0\rangle$ if and only if $\langle x,y\rangle=\langle x_{k}',\zeta_{n^k}y_{k}\rangle$, where $\zeta_{n^k}\in\Fq$ is a $n^k$th root of unity that is not a $n^{k-1}$th root of unity and $x_{k}'=\zeta_{n^k}x_k$, if $\delta\neq 0$ or $x_k$ is any element in $\Fq$ otherwise.

   Let $\zeta_{n^{k+1}}$ be a $n^{k+1}$th root of unity that is not a $n^k$th root of unity.  First, we suppose that $\delta_1\neq0$. We have that $$f(\zeta_{n^{k+1}}x_{k+1},\zeta_{n^{k+1}}y_{k+1})=\langle\zeta_{n^{k+1}}^nx_{k},\zeta_{n^{k+1}}^ny_{k}\rangle.$$ Since $(\zeta_{n^{k+1}}^n)^{n^k}=1$ and $(\zeta_{n^{k+1}}^n)^{n^{k-1}}=\zeta_{n^{k+1}}^{n^k}\neq 1$, by the induction hypothesis, it follows that $\langle\zeta_{n^{k+1}}^nx_{k},\zeta_{n^{k+1}}^ny_{k}\rangle$ is an element in the $k$th level of the tree with root $\langle x_0,y_0\rangle$. Hence, $\langle \zeta_{n^{k+1}}x_{k+1},\zeta_{n^{k+1}}y_{k+1}\rangle$ is an element in the $(k+1)$th level.  If $\delta_1=0$, then $$f(x,\zeta_{n^{k+1}}y_{k+1})=\langle0,\zeta_{n^{k+1}}^ny_{k}\rangle$$ for any $x\in\Fq$ and the same holds.

    Conversely, suppose that $$\langle x,y\rangle \in f^{-1}(x_{k}',\zeta_{n^{k}}y_{k}),$$ where $x_{k}'=\zeta_{n^{k}}x_{k}$ if $\delta\neq 0$ or $x_{k}'$ is any element in $\Fq$ otherwise. In any case, $y$ must satisfy $(-2y\beta)^{n-1}\delta_2y=\zeta_{n^{k}}y_{k}$. Moreover, by definition $(-2y_{k+1}\beta)^{n-1}\delta_2y_{k+1}=y_{k}$, which implies that $$(-2y\beta)^{n-1}\delta_2y=\zeta_{{n}^{k}}(-2y_{k+1}\beta)^{n-1}\delta_2y_{k+1}.$$ Hence,  it follows that $y^n=\zeta_{n^{k}}y_{k+1}^n$ and $y=\theta y_{k+1}$, where $\theta^n=\zeta_{n^{k}}$. Furthermore, we have that $\theta^{n^k}=\zeta_{n^{k}}^{n^{k-1}}\neq 1$ and $\theta^{n^{k+1}}=\zeta_{n^{k}}^{n^k}=1$. Hence, $\theta$ is a $n^{k+1}$th root of unity and not a $n^k$th root of unity. Therefore, if $\langle x,y\rangle$ is an element in the $k$th level of the tree with root $\langle x_0,y_0\rangle$, then $\langle x,y\rangle=\langle \zeta_{n^k}x_{k},\zeta_{n^k}y_{k}\rangle$ if $\delta_1\neq0$, or $\langle x,y\rangle=\langle x,\zeta_{n^k}y_{k}\rangle$, otherwise, where $\zeta_{n^k}\in\Fq$ is a $n^k$th root of unity that is not a $n^{k-1}$th root of unity and $x$ is any element in $\Fq$.

    Finally, for any $j> 1$ we have that $\mathfrak{g}(n^{e+j})=\mathfrak{g}(n^{e+1})=s$. Hence, any $n^{e+j}$th roots of unity are also $e^{e+1}$ roots of unity, and there are no elements in the $(e+j)$th level.


\end{proof}

In the case where $\delta_1\neq0$, Theorem \ref{tree} also holds for odd choices of $n$. Hence, the hanging trees behave in the same way. However, when $\delta_1=0$, every element with nonempty preimage has $q\cdot\mathfrak{g}(n)$ elements in its preimage instead of $\mathfrak{g}(n)$ elements. Thus, we have the following

\begin{Df}
    Let $n$ be a fixed integer with $\mathfrak{g}(n)>1$, and write $q-1=sr$, where $r$ is the greatest integer such that $\gcd(r,n)=1$. We define the tree $\T_n^0(1)$ as a single vertex with $q\cdot \mathfrak{g}(n)-1$ vertices directed to it.
    
    Let $e$ be the greatest integer such that $\mathfrak{g}(n^e)<s$. For every $1\leq i\leq e$ we define the tree $\T_n^0(i+1)$ in the following way: First we label the elements in the $i$th level of $\T_n^0(i)$ from left to right in ascending order skipping the numbers that are multiples of $q\cdot\frac{\mathfrak{g}(n^i)}{\mathfrak{g}(n^{i-1})}$. Then we attach $q\cdot\mathfrak{g}(n)$ new vertices to each vertex labeled with a multiple of $q\cdot\frac{\mathfrak{g}(n^{i})\mathfrak{g}(n)}{\mathfrak{g}(n^{i+1})}$.
\end{Df}

\begin{Th}
    Let $n$ be a fixed integer and $q-1=rs$, where $r$ is the greatest integer such that $\gcd(n,r)=1$. Let $e$ be the largest integer such that $\mathfrak{g}(n^e)<s$. If $\delta_2=0$, then each nonzero cyclic element in the functional graph of $f$ is the root of a tree isomorphic to $\T_n^0(e+1)$ 
\end{Th}
\begin{proof}
    The proof follows in the same direction as in Theorem \ref{tree}, however, in this case the roots of unity represent only the elements of the form $\langle 0,y\rangle$, as all the others have an empty preimage. 
\end{proof}

\begin{Ex}
Let $q=13$, $n=3$ and $f(X)=(X^{13}+3X)(X^{13}-X)^{2}$. We have that $\delta_1=4$ and $\delta_2=2$, hence, $\chi_1(\delta_2)=1$ and $\chi_2(\delta_2)=-1$, and it follows that $$\mathcal{N}(1)=0+2+1=3.$$
Moreover, we have that $\mathfrak{g}(n)=3=s$, thus, every tree attached to a nonzero periodic element is isomorphic to $\T_3(1)$. Therefore, we have one component isomorphic to 
\begin{figure}[H]
	\begin{center}
		\includegraphics[scale=0.5]{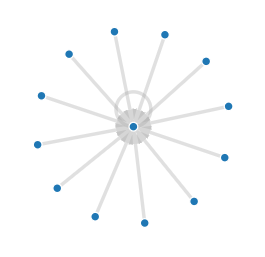}
	\end{center}
	\caption{Connected component that contains zero in the functional graph of $f(X)=(X^{13}+3X)(X^{13}-X)^{2}$ over $\F_{169}$.}
\end{figure}
which is the component that contains zero, and two components isomorphic to 
\begin{figure}[H]
	\begin{center}
		\includegraphics[scale=0.5]{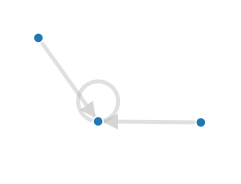}
	\end{center}
	\caption{Connected component that contains a fixed element in the functional graph of $f(X)=(X^{13}+3X)(X^{13}-X)^{2}$ over $\F_{169}$.}
\end{figure}

For the greater cycles, we first calculate $\tau_0$. Let $\beta\in\Fqd$ be such that $\beta^2=2$ and take $\alpha=2$ a generator of $\Fq^*$. Then,  $4\beta^2\delta_2=2^3$, and we get that $l_0=4$. Since $\gcd(n-1,l_0)=1$ we have $\tau_0=1$. We start with the cycles of nonzero elements with zero in the first coordinate. If $t=1$, then $t\tau_0=1$, hence we start with $t=2$. Therefore, we obtain $$\mathcal{N}(2)=\frac{1}{2}[-2+4]=1$$ and there is one component isomorphic to 
\begin{figure}[H]
	\begin{center}
		\includegraphics[scale=0.5]{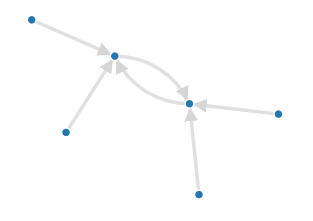}
	\end{center}
	\caption{Connected component that contains a cycle of size $2$ in the functional graph of $f(X)=(X^{13}+3X)(X^{13}-X)^{2}$ over $\F_{169}$.}
\end{figure} Any other choice of $t$ results in zero.

Now for the cycles of elements with the first coordinate different than zero we have $\tau=12$, as $\frac{\delta_1}{\delta_2}=2$. Hence, taking $t=1$ we have $$\mathcal{N}(12)=\frac{48}{12}=4$$ and there are $4$ components isomorphic to \begin{figure}[H]
	\begin{center}
		\includegraphics[scale=0.5]{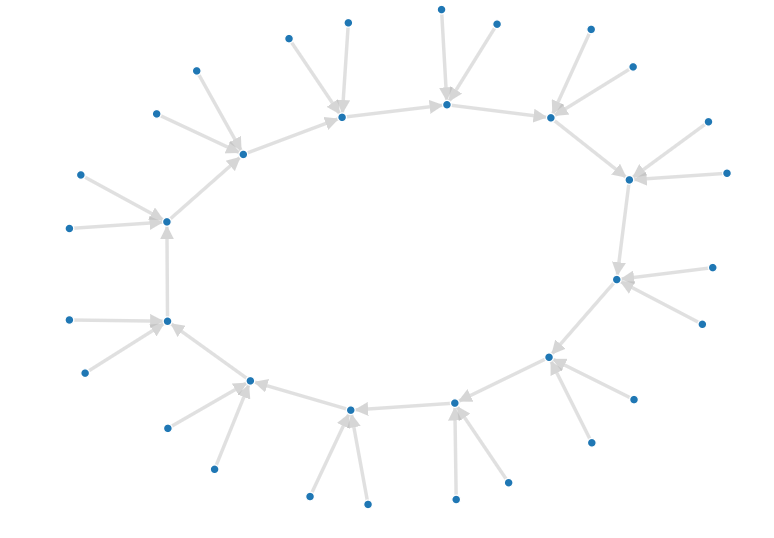}
	\end{center}
	\caption{Connected component that contains a cycle of size $12$ in the functional graph of $f(X)=(X^{13}+3X)(X^{13}-X)^{2}$ over $\F_{169}$.}
\end{figure} and those are the only components in the functional graph of $f$.
\end{Ex}
\begin{Ex}
Let $q=13$, $n=5$ and $f(X)=(X^{13}+3X)(X^{13}-X)^{4}$. Again, we have $\delta_1=4$ and $\delta_2=2$. However, $\chi_2(\delta_2)=-1$, hence $\mathcal{N}(1)=0+0+1=1$ and the only cycle of length one is the one that contains zero, which is isomorphic to
\begin{figure}[H]
	\begin{center}
		\includegraphics[scale=0.5]{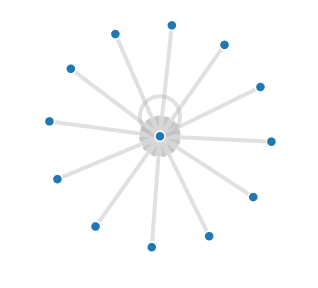}
	\end{center}
	\caption{Connected component that contains zero in the functional graph of $f(X)=(X^{13}+3X)(X^{13}-X)^{4}$ over $\F_{169}$.}
\end{figure}

Moreover, we have that $\mathfrak{g}(n)=1$, thus there are no hanging trees in the nonzero cyclic elements.

For the cycles greater than $1$, we first calculate $\tau_0$. Again, taking $\beta^2=2$ and $\alpha=2$, we obtain $l_0=7$ and $\tau_0=4$. Hence, for the cycles of elements with zero in the first coordinate, taking $t=1$ we have $$\mathcal{N}(4)=\frac{12}{4}=3$$ and there are $3$ components isomorphic to \begin{figure}[H]
	\begin{center}
		\includegraphics[scale=0.5]{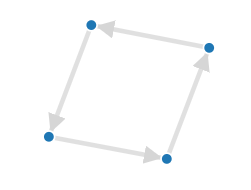}
	\end{center}
	\caption{Connected component that contains a cycle of length $4$ in the functional graph of $f(X)=(X^{13}+3X)(X^{13}-X)^{4}$ over $\F_{169}$.}
\end{figure} For any other choice of $t$ we get $0$.
Since $\frac{\delta_1}{\delta_2}=2$, we have $\tau=12$ and, taking $t=1$, we have $$\mathcal{N}(12)=\frac{144}{12}=12,$$ hence,there are $12$ components isomorphic to \begin{figure}[H]
	\begin{center}
		\includegraphics[scale=0.5]{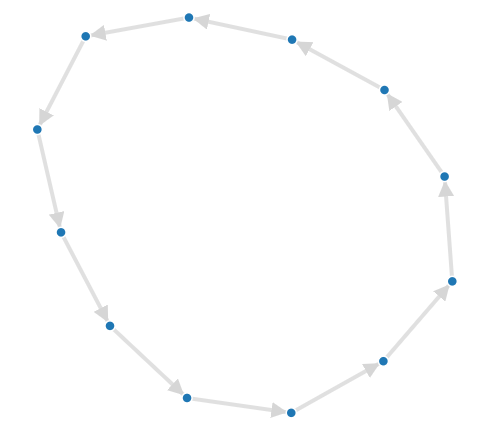}
	\end{center}
	\caption{Connected component that contains a cycle of length $12$ in the functional graph of $f(X)=(X^{13}+3X)(X^{13}-X)^{4}$ over $\F_{169}$.}
\end{figure} and zero for any other choice of $t$. Therefore, those are the only components in the functional graph of $f$. 
\end{Ex}


\end{document}